\documentclass{amsart}
\usepackage[utf8]{inputenc}
\usepackage{amsmath}
\usepackage{amssymb}
\usepackage{caption}
\usepackage{amsthm}
\usepackage{mathtools}
\usepackage{enumerate}
\usepackage{graphicx,tikz}
\usepackage[colorlinks=true, linkcolor=blue, citecolor=blue]{hyperref}
\usepackage{enumitem}
\usepackage{comment}

\usepackage[all,2cell,ps]{xy}

\newtheorem{theorem}{Theorem}[section]
\newtheorem{corollary}[theorem]{Corollary}
\newtheorem{lemma}[theorem]{Lemma}
\newtheorem{proposition}[theorem]{Proposition}

\theoremstyle{definition}
\newtheorem{definition}[theorem]{Definition}
\newtheorem{remark}[theorem]{Remark}

\newtheoremstyle{TheoremNum}
{.5em}{.5em}              
{\itshape}                      
{}                              
{\bfseries}                     
{.}                             
{ }                             
{\thmname{#1}\thmnote{ \bfseries #3}}
\theoremstyle{TheoremNum}
\newtheorem{thmn}{Theorem}

\newcommand{\bbC}{\mathbb{C}}

\newcommand{\bbR}{\mathbb{R}}

\newcommand{\bbZ}{\mathbb{Z}}

\newcommand{\calH}{\mathcal{H}}

\newcommand{\calM}{\mathcal{M}}

\newcommand{\calO}{\mathcal{O}}

\DeclareMathOperator{\Orth}{O}

\DeclareMathOperator{\PSL}{PSL}

\DeclareMathOperator{\PO}{PO}
\DeclareMathOperator{\SO}{SO}
\DeclareMathOperator{\PSO}{PSO}

\DeclareMathOperator{\POm}{P\Omega}

\def\F{{\mathbb F}}

\DeclareMathOperator{\Isom}{Isom}
\DeclareMathOperator{\Aut}{Aut}

\DeclareMathOperator{\Out}{Out}

\DeclareMathOperator{\ord}{ord}

\makeatletter
\let\@@pmod\pmod
\DeclareRobustCommand{\pmod}{\@ifstar\@pmods\@@pmod}
\def\@pmods#1{\mkern4mu({\operator@font mod}\mkern 6mu#1)}
\makeatother

\usetikzlibrary{arrows}

\renewcommand{\qed}{\hfill$\scriptstyle\blacksquare$}

\title[Normalizers and isometry groups of arithmetic manifolds]{Normalizers of lattices and isometry groups of arithmetic hyperbolic manifolds}

\dedicatory{Dedicated to Alex Lubotzky on his 70th birthday}

\author[M. Belolipetsky]{Mikhail Belolipetsky}
\address{IMPA, Estrada Dona Castorina 110, 22460-320 Rio de Janeiro, Brazil}
\email[]{mbel@impa.br}

\author[T. Cheetham-West]{Tam Cheetham-West}
\address{Department of Mathematics, Yale University, New Haven, CT, 06511}
\email[]{tamunonye.cheetham-west@yale.edu}

\begin{document}
\date{\today}
\begin{abstract}
We prove that every arithmetic lattice in $\PSL(2,\bbC)$ and every arithmetic lattice of the simplest type in $\PO(n,1)$, $n\ge 2$, is the normalizer of arbitrarily many of its sublattices. Combined with previous work, this result implies that every lattice in $\PSL(2,\bbC)$ has this property. In this way, we prove that the set of profinitely flexible lattices in $\PSL(2,\bbC)$ is either empty or countably infinite. Another result is that every finite group is realized as the full isometry group of an arithmetic hyperbolic $n$-manifold. The proof of this theorem is based on study of normalizers of lattices and subgroup growth theory.
\end{abstract}

\maketitle

\section{Introduction}
An arithmetic lattice $\Lambda$ in a semisimple Lie group $H$ has a property that its commensurator $\mathrm{Comm}_H(\Lambda)$ is dense in $H$ up to a finite index \cite{Bor66}. From the geometric point of view it shows that the associated locally symmetric quotient space has infinitely many hidden symmetries that correspond to the elements of the normalizers in $H$ of the finite index subgroups of $\Lambda$. This makes the problem of controlling the normalizers and associated isometry groups much more challenging than in the non-arithmetic case. Not surprisingly, several previously known results were obtained under an extra assumption that the lattices to which they apply should be non-arithmetic (see e.g. \cite{ACM, BelLubotzkyIsom}). In this paper we will show that this assumption is not always necessary and extend the results of \cite{ACM} and \cite{BelLubotzkyIsom} to a much wider class of lattices.

Let $H = \PO(n,1)$ be the group of isometries of the $n$-dimensional hyperbolic space $\calH^n$, and let $\Lambda\cong\pi_1(\calM)$, where $\calM$ is an arithmetic hyperbolic $n$-manifold. There are two situations where we want to understand the behavior of the normalizers in $H$ for arithmetic lattices:
\begin{enumerate}
    \item To show that a given finite-index subgroup $\Delta<\Lambda$ has infinitely many sublattices whose full normalizer in $H$ is the grpup $\Lambda$.
    \item To show that in every dimension $n\geq 2$, every finite group is the full group of isometries for infinitely many arithmetic hyperbolic $n$-manifolds. 
\end{enumerate}

In Section~\ref{sec:normalizers}, following up on the work by Agol, Cheetham-West, and Minsky in~\cite{ACM}, we prove:  

\begin{thmn}[\ref{mainthm1}]
Let $\Lambda$ be a torsion-free arithmetic lattice in $\PO(3,1)$ or a torsion-free arithmetic lattice of the simplest type in $\PO(n,1)$, $n\ge 2$.
For a finite-index subgroup $\Delta<\Lambda$, there is a finite-index subgroup $\Theta<\Delta$ such that $\Aut(\Theta)=\Lambda$. 
\end{thmn}

Here $\Aut(\Delta)$ denotes the automorphisms of the lattice $\Delta$. It is a group $\Aut(\Delta) = N_H(\Delta)$, the normalizer of the lattice $\Delta$ in $H = \PO(n,1)$. The same result also holds for $\Aut^+(\Delta)$ which denotes the group of orientation-preserving automorphisms of $\Delta$.
The main tool and the context for the theorem and its corollaries given below is the good control over images (in finite quotients) of independent elements satisfied by (relatively) word-hyperbolic groups that act properly and cocompactly on $\mathrm{CAT}(0)$ cube complexes. These cubulated groups satisfy the property of {\it omnipotence} in the sense of Wise, which is known elsewhere in the literature as the {\it strong commanding property}. This property can be thought of as encoding and quantifying the conjugacy separability properties of cubulable (relatively) hyperbolic groups. 
 
The second main result of the paper is about the isometry groups of arithmetic hyperbolic $n$-manifolds. In Section~\ref{sec:isometries} we extend the subgroups counting method of Belolipetsky and Lubotzky in \cite{BelLubotzkyIsom} to prove:

\begin{thmn}[\ref{mainthm2}]
For every $n\geq 2$ and every finite group $G$ there exist infinitely many arithmetic hyperbolic $n$-manifolds $\mathcal{M}$ of the simplest type with $\Isom(\mathcal{M})\cong G$. 
\end{thmn}

The main challenge in proving the theorems is dealing with the dense commensurator of an arithmetic lattice $\Lambda$. As the first step, we show that although there are infinitely many maximal lattices in $\mathrm{Comm}_H(\Lambda)$, only finitely many of them contain $\Lambda$ (see Proposition~\ref{finitelymany}). After this we are left with dealing with a delicate problem of controlling the normalizers of the sublattices of $\Lambda$. Here the idea is to force them to be restricted to the finite set of the maximal lattices.

For finitely generated, residually finite groups the property of \emph{profinite rigidity}, i.e. being determined up to isomorphism by the set of all actions on finite sets, is not well-behaved with respect to taking finite-index subgroups or finite group extensions. Famously, Baumslag's nilpotent examples $(\bbZ/25\bbZ)\rtimes_\alpha \bbZ$ and $(\bbZ/25\bbZ)\rtimes_{\alpha^2}\bbZ$ \cite{Baumslag} (where $\alpha$ is multiplication by $6$ (mod $25$)) are virtually infinite cyclic non-isomorphic groups with isomorphic profinite completions. In contrast, Bridson and Reid \cite{BRPrasad} proved that if a lattice in $\PSL (2,\bbC) \cong \Isom^+(\calH^3)$ is profinitely rigid, so also is its normalizer in $\PSL (2,\bbC)$. That this normalizer is a finite index group extension is a consequence of the Mostow--Prasad rigidity theorems \cite{Mostow}\cite{Prasad}.

Combining Theorem~\ref{mainthm1}, the main result of \cite{ACM}, and the normalizer inheritance theorem of Bridson--Reid \cite[Theorem~4.4]{BRPrasad}, we obtain the following corollary:

\begin{corollary}\label{cor1}
The set of profinitely flexible lattices in $\PSL(2,\bbC)$ is either empty or countably infinite.
\end{corollary}

As another corollary of Theorem~\ref{mainthm1} and the main result of \cite{ACM}, we have:
\begin{corollary}\label{cor2}
    The following are equivalent:
    \begin{enumerate}
        \item Every lattice in $\PSL(2,\bbC)$ is profinitely rigid.
        \item Every fibered lattice in $\PSL(2,\bbC)$ is profinitely rigid.
        \item Every special lattice in $\PSL(2,\bbC)$ is profinitely rigid. 
    \end{enumerate}
\end{corollary}

Notice that these corollaries for non-arithmetic lattices were already obtained in~\cite{ACM}. Combining the results of \cite{ACM} and the present paper allows us to remove all the arithmeticity related assumptions.  

The paper is organized as follows. In Section~\ref{sec prelims} we collect some results about arithmetic groups, omnipotence, and geodesics of hyperbolic manifolds to be used later on. In Section~\ref{sec:normalizers} we prove Theorem~\ref{mainthm1} and discuss the corollaries. In Section~\ref{sec:isometries} we prove Theorem~\ref{mainthm2}.
\medskip

\noindent
\textbf{Acknowledgments.} 
We thank Ian Agol whose questions and comments initiated this project. T.C. is grateful to Yair Minsky for his encouragement of this project. We thank Matthew Stover for his comments on a preliminary version of the paper.

Part of this work was completed when the authors visited the Simons Laufer Mathematical Institute in March 2026 for the {\it Topological and Geometric Structures in Low Dimensions} program. The work of M.B. is partially supported by the CNPq grant 300233/2025-6 and the FAPERJ grant E-26/204.250/2024.

\section{Preliminaries}\label{sec prelims}
\subsection{Arithmetic subgroups.} A \emph{lattice} in a Lie group $H$ is a discrete subgroup of finite covolume. 
Lattices in $H = \Isom(\calH^n)$ can be constructed using number theory.

Recall that the isometry group $H$ can be identified with $\Orth_0(n, 1)$, the subgroup of the orthogonal group $\Orth(n, 1)$ which preserves the upper-half space. It is isomorphic to the projective orthogonal group $\PO(n, 1) =\Orth(n, 1)/\{+1,-1\}$. The subgroup $\SO_0(n, 1)$ of $\Orth_0(n, 1)$ of all the elements with determinant $1$ is the group of orientation preserving isometries.

Let $k$ be a totally real number field with the ring of integers $\calO$. Consider a quadratic form $f$ of signature $(n,1)$ defined over $k$ such that for every non-identity embedding $\sigma : k\to\bbR$ the form $f^\sigma$ is positive definite. Let $\Gamma = \Orth_0(f,\calO)$ be the subgroup of the integral automorphisms of $f$ in  $H$. By a classical theorem of Borel and Harish-Chandra, the group $\Gamma$ is a lattice in $H$.  Lattices obtained in this way and all the subgroups of $H$ which are commensurable with them are called \emph{arithmetic lattices of the simplest type} or \emph{arithmetic lattices of type~I}. There are also type~II and type~III arithmetic lattices in $H$ but we will not consider them in this paper except when $n = 3$ (see \cite{BBKS} for more details on the classification of arithmetic subgroups of $\PO(n,1)$). A detailed study of arithmetic lattices in dimension $3$ is presented in \cite{MR}.


An important property that characterizes arithmetic hyperbolic manifolds of the simplest type in dimension $n\ge3$ is that they contain infinitely many totally geodesic hypersurfaces (cf.~\cite{BBKS}). This allowed Bergeron, Haglund and Wise \cite{BHW} to show that the simplest type lattices in $H$ are virtually special cubulable groups, a property which implies omnipotence that we discuss in the next subsection.  

\subsection{Omnipotence.}
We will need a few definitions. 
\begin{definition}
A finite subset $\{ g_1,\dots,g_m\}$ of infinite order elements in a group $\Gamma$ is called \emph{independent} if no non-zero power of $g_i$ is conjugate to a non-zero power of $g_j$ for $i \neq j$.
\end{definition}
\begin{definition}
Let $\Gamma > \Delta$ be a group and a finite-index subgroup. For $\gamma\in \Gamma$, the \emph{$\Delta$-degree} of $\gamma$ is the smallest positive integer $K$ for which $\gamma^{K}\in \Delta$. 
\end{definition}
\begin{definition}
A group $\Gamma$ is called \emph{omnipotent} if for any independent subset $\{ g_1,\dots,g_m\}$ there is a constant $\kappa$ so that for any $m$-tuple of positive integers $(N_1,\dots,N_m)$, there is a finite quotient $q:\Gamma\twoheadrightarrow Q$ with $\ord(q(g_i))=\kappa N_i$ for $1\leq i\leq m$. 
\end{definition}

The notion of omnipotence was first defined by Wise in \cite{omnipotence} where he proved that free groups have this property \cite[Theorem~3.5]{omnipotence}. In \cite{WiseHaken}, Wise used the Malnormal Special Quotient Theorem \cite[Theorem~12.2]{WiseHaken} to prove the following result.

\begin{theorem}[{\cite[Theorem 14.26]{WiseHaken}}]\label{WiseOmnipotent}
    Let $\Gamma$ be a virtually special, word-hyperbolic group. Then $\Gamma$ is omnipotent. 
\end{theorem}

This theorem is sufficient to cover the cocompact case in our Theorem~\ref{mainthm1}. For the non-cocompact case we recall the result of Shepherd \cite{goodsam} who proved a general omnipotence statement for finite independent subsets of convex elements in virtually special cubulated groups. We recall that a cubulated group $\Gamma$ is a finitely generated group together with a geometric action on a $\mathrm{CAT}(0)$ cube complex $X$ by cubical automorphisms. We refer to \cite{WiseHaken, goodsam} for the definition of a virtually special cubulated group. Following \cite[Definition~2.19]{goodsam}, a subgroup $K$ of a cubulated group $\Gamma \curvearrowright X$ is called \emph{convex} if it stabilises a convex subcomplex $Y\subset X$ with finite quotient $Y/K$. An element $g \in \Gamma$ is convex if $\langle g \rangle$ is a convex subgroup of $\Gamma$.

\begin{theorem}[{\cite[Theorem 1.2]{goodsam}}]\label{sammy}
Let $\Gamma$ be a virtually special cubulated group. Then for any independent subset $\{ g_1,\dots,g_m\}$ of convex elements in $\Gamma$ there is a constant $\kappa$ so that for any $m$-tuple of positive integers $(N_1,\dots,N_m)$, there is a finite quotient $q:\Gamma\twoheadrightarrow Q$ with $\ord(q(g_i))=\kappa N_i$ for $1\leq i\leq m$. 
\end{theorem}
For the proof of Theorem~\ref{mainthm1}, we will use the following theorem for cubulable relatively hyperbolic groups.
\begin{theorem}\label{thm:omnirelative}
    Let $(\Gamma,\mathcal{P})$ be a relatively hyperbolic virtually special group. For any independent subset $\{g_1,\dots,g_m\}$ of hyperbolic elements in $\Gamma$ and for any $m$-tuple of positive integers $(N_1,\dots,N_m)$, the group $\Gamma$ has a hyperbolic virtually special quotient $\overline{\Gamma}$ such that the kernel of $\Gamma\twoheadrightarrow\overline{\Gamma}$ contains $\langle\langle g_1^{\kappa N_1},\dots,g_m^{\kappa N_m}\rangle\rangle$ for some positive integer $\kappa$ (depending only on $g_1,\dots,g_m$) and the images of $g_1,\dots,g_n$ in $\overline{\Gamma}$ have order $\kappa N_1,\dots,\kappa N_n$.
\end{theorem}
    \begin{proof}
    In the hyperbolic case, this is a consequence of the Malnormal Special Quotient Theorem \cite[Theorem~12.2]{WiseHaken} and the proof of Theorem~\ref{WiseOmnipotent}. We focus on the relatively hyperbolic virtually special case. Let $g_1,\dots,g_m$ be the set of independent hyperbolic elements in the theorem statement. A maximal elementary subgroup containing $g_i$ is a maximal two-ended subgroup of $\Gamma$ containing $g_i$. This maximal elementary subgroup is unique \cite[Lemma 5.5]{OsinRelHyp} because $g_i$ is hyperbolic and is denoted by $E_\Gamma(g_i)$. We first apply \cite[Lemma 5.5]{OsinRelHyp} inductively (using independence) to see that $(\Gamma,\mathcal{P}\cup_i E_{\Gamma}(g_i))$ is relatively hyperbolic. By performing a deep enough Dehn filling on the parabolic subgroups of $\Gamma$ (subgroups in $\mathcal{P}$ and $E_{\Gamma}(g_i))$, we can choose a hyperbolic virtually special quotient $\rho:\Gamma\twoheadrightarrow \Gamma'$ (by \cite[Theorem 2]{Einstein}) where the $\rho(g_i)$ elements have finite order $\kappa N_1,\dots,\kappa N_m$ for some $\kappa$ depending only on $g_1,\dots,g_m$.
    \end{proof}

The fundamental groups of finite volume hyperbolic $3$-manifolds are  cubulated and virtually special by the work of Agol and Wise \cite{AgolHaken, WiseHaken}. By Bergeron--Haglund--Wise this also applies to the groups of arithmetic hyperbolic $n$-manifolds of the simplest type \cite{BHW}.

\subsection{Simply transitive geodesics.}
Our application of omnipotence is based on a geometric property of closed geodesics of a hyperbolic manifold:

\begin{lemma}[see {\cite[Lemma~3.2]{ACM}}]\label{freeorbitlemma}
    Let $\calM$ be a finite-volume orientable hyperbolic $n$-manifold. There are infinitely many distinct closed geodesics in $\calM$ which are not fixed by any non-trivial isometry. 
\end{lemma}

Here the geodesic $\gamma$ is said to be \emph{fixed} by an isometry $f$ if $f(\gamma) = \gamma$. Notice that it is not required that $\gamma$ is fixed by $f$ pointwise. 
The proof of the lemma in \cite{ACM} is based on counting of closed geodesics of bounded length and shows that a positive proportion of them has this property.

From the group theoretic viewpoint, a closed geodesic $\gamma$ in $\calM = \Gamma\backslash\calH^n$ corresponds to a primitive hyperbolic element $g\in \Gamma$. The geodesic $\gamma$ is fixed by an isometry of $\calM$ corresponding to an automorphism $\varphi\in\Aut(\Gamma)$ if $\varphi(g) = h g^k h^{-1}$ for some $k\in \mathbb{Z}\setminus\{0\}$ and $h\in \Gamma$. Thus, if $\phi(g)$ is not commensurable with $g$ in the sense of \cite[Definition~4.1]{osinminaysan}, then $\gamma$ is not fixed by $\varphi$.

\begin{remark}
Generalizing the previous result of Lubotzky for free groups \cite{Lu1}, Minasyan and Osin proved that an automorphism $\varphi:G\to G$ that represents a non-trivial element in $\Out(G)$ of a non-elementary relative hyperbolic group $G$ without finite normal subgroups
is not commensurating \cite[Corollary 1.4]{osinminaysan}. In our case it shows that for any non-trivial isometry of $\calM$ there is a closed geodesic that is not fixed. It may be possible to use the methods of \cite{osinminaysan} to show that there exists a closed geodesic that is not fixed by \emph{all} non-trivial isometries and, moreover, that there are \emph{infinitely many} such geodesics. This would give an alternative proof of Lemma~\ref{freeorbitlemma}, but we will not pursue it here. 
\end{remark}

\section{Normalizers of arithmetic lattices} \label{sec:normalizers}

\begin{proposition}\label{finitelymany}
Let $\Lambda$ be a lattice in $\PO(n,1)$. There are only finitely many maximal lattices $\Gamma_1,\dots,\Gamma_m$ with $\Lambda<\Gamma_i$ for $i=1,\dots,m$.  
\end{proposition}

\begin{proof}
 When $\Lambda$ is non-arithmetic, then the commensurator of $\Lambda$ in $\PO(n,1)$ is the unique maximal lattice containing $\Lambda$ \cite[Theorem 1, p.~2]{margDisc}. When $\Lambda$ is arithmetic, by Borel's Volume Formula \cite{Borel} (see also \cite[Theorem 11.2.1]{MR} for example), there are at most finitely many lattices up to conjugation containing $\Lambda$. We can reduce each conjugacy class to a collection $\{g_\alpha Hg_\alpha^{-1}\}_{\alpha\in\lambda}$ of maximal lattices with $\Lambda< \bigcap_{\alpha\in\lambda} (g_\alpha Hg_\alpha^{-1})$. Each $g_\alpha$ gives a conjugation isomorphism $C_{g_\alpha}:H\to g_\alpha Hg_\alpha^{-1}$. Since $H$ is finitely generated as a lattice, the set $\{C_{g_\alpha}^{-1}(\Lambda)\}_{\alpha\in\lambda}$ is finite, so we can pass to a subset $\lambda'\subset\lambda$ of lattices where $\{C_{g_\alpha}^{-1}(\Lambda)\}_{\alpha\in\lambda'}$ is a single fixed lattice $\Lambda'<H$. For $\beta,\alpha\in\lambda'$, $C_{g_\beta}^{-1}\circ C_{g_\alpha}|_{\Lambda'}:\Lambda'\to\Lambda'$ is an automorphism. Since $\Out(\Lambda')$ is finite, by the Mostow--Prasad rigidity theorem \cite{Mostow}\cite{Prasad} the set of elements $\{g_\beta^{-1}g_\alpha\}_{\alpha,\beta\in\lambda'}\in \Aut(\Lambda')\cong N(\Lambda')$ fall into finitely many outer classes. Repeating the foregoing argument for $\Lambda''$, a finite-index subgroup of $\Lambda'$ with $N(\Lambda'')\cong H$, for all $\alpha,\beta\in\lambda'$ there are elements $n_{\alpha\beta}$ of $H$ such that $g_\alpha=g_\beta n_{\alpha\beta}$. It follows that $g_\alpha H g_\alpha^{-1}=g_\beta H g_\beta^{-1}$ and there are at most finitely many distinct conjugates of $H$ containing $\Lambda''$. So there are at most finitely many lattices containing $\Lambda$. 
\end{proof}

\begin{proposition}\label{base case} 
Let $\Lambda$ be a torsion-free arithmetic lattice in $\PO(3,1)$ or a torsion-free arithmetic lattice of the simplest type in $\PO(n,1)$.
For any finite-index subgroup $\Delta<\Lambda$ and any maximal lattice $\Gamma>\Lambda$ there is a lattice $\Theta<\Delta<\Lambda$ such that $\Aut(\Theta) \cap \Gamma =\Lambda$.  
\end{proposition}

\begin{proof}
For a finite index subgroup $\Delta<\Lambda$, we choose $\Delta'<\Delta<\Lambda$ a finite-index subgroup with $\Aut(\Delta')=\Gamma$. We also choose a hyperbolic element $\gamma$ whose conjugacy class corresponds to a geodesic in $\Delta'\backslash\calH^n$ on which the group $\Gamma/\Delta'\cong \Out(\Delta')$ acts simply transitively by Lemma~\ref{freeorbitlemma}. 

For $\{g_1,\dots,g_k\} = \Lambda/\Delta'$ and $\{g_{k+1},\dots,g_l\} = (\Gamma/\Delta')\smallsetminus (\Lambda/\Delta')$, we choose preimages $g_i'\in \Gamma$ for $1\leq i\leq l$ with $g_1'=1\in\Gamma$. We consider the non-conjugate (in $\Delta'$) cyclic subgroups $\langle g_i'\gamma g_i'^{-1} \rangle$ for $1\leq i\leq k$ and $\langle g_j'\gamma g_j'^{-1} \rangle$ for $k+1\leq j\leq l$. By construction, these subgroups intersect all cusp subgroups of $\Delta'$ trivially. They are therefore convex subgroups of $\Delta'$ as defined in \cite{goodsam} (see Definition~2.19). Since $\Delta'$ is cubulated and virtually special \cite{AgolHaken, BHW, WiseHaken}, by Theorem~\ref{sammy}, $\Delta'$ is omnipotent and so there is an integer $\kappa$ and a finite quotient $\rho:\Delta'\rightarrow Q$ such that $\ord(\rho(g_i'\gamma g_i'^{-1}))=\kappa N_1$ for $1\leq i\leq k$, $\ord(\rho(g_j'\gamma g_j'^{-1}))=\kappa N_2$ for $k+1\leq j\leq l$, and $N_1\ne N_2$. 

Set $\Theta=\ker(\rho)$. For $f\in \Gamma\smallsetminus\Lambda$, if $f \Theta f^{-1}=\Theta$, then $f$ induces an automorphism $$\Tilde{f}:\Delta'/\Theta\to\Delta'/\Theta,$$  
which sends the image $\tilde{g_i'}$ of a representative of a 
conjugacy class representing $\{g_i'\gamma g_i'^{-1}\}$ with $1\leq i\leq k$ to the image $\tilde{g_j'}$ of a representative of a conjugacy class representing $\{g_j'\gamma g_j'^{-1}\}$ with $k+1\leq j\leq l$. To see this, notice that $f=g_t' h$ for some $h\in \Delta'$ and $1 \le t \le l$ .
Since $f\in \Gamma\smallsetminus\Lambda$, $t\geq k+1$. 
Now the conjugation by $f$ that fixes $\Theta$ induces a map $\tilde{f}$ that sends $\tilde{g_1'}$ to $\tilde{g_t'}$, and by the construction the orders of $\tilde{g_1'}$ and $\tilde{g_t'}$ in $\Delta'/\Theta$ are distinct in yielding a contradiction.  
 
Thus, $\Delta'<\Aut(\Theta)<\Lambda$. If $\Theta$ is $\Lambda$-invariant, then $\Aut(\Theta)=\Lambda$, and we are done. Otherwise, consider the $\Lambda$-invariant subgroup $\Theta'=\bigcap_{i=1}^k g_i\Theta g_i^{-1}$. 

 By construction, $\Delta'/\Theta'\cong \prod_{i\in S}\Delta'/ g_i\Theta g_i^{-1}$, where $S\subset \{1, \ldots, k\}$ is a non-empty subset. We can check that the homomorphism $\rho':\Delta'\to \Delta'/\Theta'$ also satisfies $\ord(\rho'(g_i\gamma g_i^{-1}))=\kappa N_1$ for $1\leq i\leq k$ and $\ord(\rho'(g_j\gamma g_j^{-1}))=\kappa N_2$ for $k+1\leq j\leq l$. To see that this is true, let $1\leq i,i'\leq k$ and let $k+1\leq j\leq l$. The order of the image of $g_{i'}\gamma g_{i'}^{-1}$ in $\Delta'/g_i\Theta g_i^{-1}$ is the same as the order of the image of $g_i^{-1}(g_{i'}\gamma g_{i'}^{-1})g_i$ in $\Delta'/\Theta$ which is $\kappa N_1$ by the previous paragraph. Thus, $\ord(\rho'(g_{i'}\gamma g_{i'}^{-1}))$ is the least common multiple of the orders of the images of $g_{i'}\gamma g_{i'}^{-1}$ in each factor of $\prod_{i\in S}\Delta'/ g_i\Theta g_i^{-1}$, and so $\ord(\rho'(g_{i'}\gamma g_{i'}^{-1}))=\kappa N_1$. The order of the image of $g_{j}\gamma g_{j}^{-1}$ in $\Delta'/g_i\Theta g_i^{-1}$ is the same as the order of the image of $g_i^{-1}(g_{j}\gamma g_{j}^{-1})g_i$ in $\Delta'/\Theta$ which is $\kappa N_2$ by the previous paragraph. So, $\ord(\rho'(g_{j}\gamma g_{j}^{-1}))$ is the least common multiple of the orders of the images of $g_{j}\gamma g_{j}^{-1}$ in each factor of $\prod_{i\in S}\Delta'/ g_i\Theta g_i^{-1}$, and so $\ord(\rho'(g_{j}\gamma g_{j}^{-1}))=\kappa N_2$ as claimed. Thus, $\Theta'$ is a $\Lambda$-invariant finite-index subgroup of $\Delta'$ with no other automorphisms in $\Gamma$ outside $\Lambda$ and $\Aut(\Theta')\cap \Gamma =\Lambda$ as claimed. 
\end{proof}

We can now prove the first main result.

\begin{theorem}\label{mainthm1}
Let $\Lambda$ be a torsion-free arithmetic lattice in $\PO(3,1)$ or a torsion-free arithmetic lattice of the simplest type in $\PO(n,1)$.
For a finite-index subgroup $\Delta<\Lambda$, there is a finite-index subgroup $\Theta<\Delta$ such that $\Aut(\Theta)=\Lambda$. 
\end{theorem}
\begin{proof}
By Proposition~\ref{finitelymany}, there are finitely many maximal lattices $\Gamma_1, \dots, \Gamma_m$ containing $\Lambda$. We will proceed by induction in the number $m$ of these lattices. Using Proposition~\ref{base case}, we can find $\Theta_1<\Delta<\Lambda$ such that $\Aut(\Theta_1)\cap \Gamma_1=\Lambda$. This is the base case for the induction. Assuming we have a lattice $\Theta_k<\Delta<\Gamma_{k}$ for $k\in \{1,\dots,m\}$, such that $\Aut(\Theta_k)\cap\Gamma_i=\Lambda$ for $i=1,\dots,k$ and $k < m$, we want to show that we can find a lattice $\Theta_{k+1}$ such that $\Aut(\Theta_{k+1})\cap\Gamma_i=\Lambda$ for $i=1,\dots,k+1$: 
$$
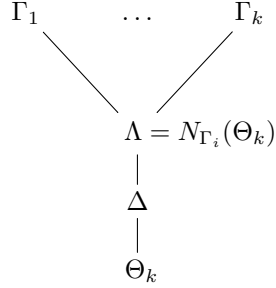
\xymatrix@C=2pt@R=2.8ex{
*+[c]{\Gamma_1} \ar@{-}[ddr] & \ldots & {\Gamma_k} \ar@{-}[ddl]\\
& & \\
                                &*+[r]{\Lambda = N_{\Gamma_i}(\Theta_k)}\ar@{-}[d] & \\
                                &{\Delta}\ar@{-}[d] & \\
                                &{\ \Theta_k} & \\
}
$$

Fix a $\Lambda$-invariant subgroup $\Theta_k<\Delta$ as above.
For each $j=1,\dots,k$ by Lemma~\ref{freeorbitlemma} we can choose a hyperbolic element $\gamma_j\in \Lambda$  such that for all elements $g_{lj}\in \Gamma_j\smallsetminus\Lambda$ in a left transversal for $\Lambda$ in $\Gamma_j$, $g_{lj}\gamma_jg_{lj}^{-1}$ is not conjugate to $\gamma_j$ in $\Lambda$. At this step, we observe that $g_{lj}$ normalizes $\Lambda$ if and only if $g_{lj}\gamma_jg_{lj}^{-1} \in \Lambda$. We can then find a common power $N_j$ such that for all $g_{lj}\in \Gamma_j$ in the left transversal for $\Lambda$ in $\Gamma_j$, $(g_{lj}\gamma_jg_{lj}^{-1})^{N_j}\in\Theta_k$.
  
We then choose a $\Gamma_{k+1}$-invariant subgroup of $\Theta_k$ which we denote by $\Delta_{k+1}$, and we choose a common power $N$ such that 
\begin{equation}\label{eq-spade}
 \langle g_{11}\gamma_1^{NN_1}g_{11}^{-1}\rangle,\dots,\langle g_{lk}\gamma_k^{NN_k}g_{lk}^{-1}\rangle 
\end{equation}
is a family of malnormal cyclic subgroups of $\Delta_{k+1}$. Now apply Lemma~\ref{freeorbitlemma} to choose a hyperbolic element $\alpha\in \Delta_{k+1}$ such that for every element $h\in \Gamma_{k+1}\smallsetminus \Delta_{k+1}$, $h\alpha h^{-1}$ is not conjugate to $\alpha$ in $\Delta_{k+1}$. It is also important that $\alpha$ is not conjugate into the subgroups $\langle \gamma_1^N\rangle,\dots,\langle \gamma_k^N\rangle$ as well as the finitely many distinct conjugates specified in~\eqref{eq-spade}. By choosing $h_1,\dots,h_t$, $h_{t+1},\dots,h_l\in \Gamma_{k+1}$ to represent all the distinct cosets of $\Gamma_{k+1}/\Delta_{k+1}$ (a left transversal including $1$ of $\Delta_{k+1}$ in $\Gamma_{k+1}$) with $h_1,\dots,h_t$ representing the cosets of $\Lambda/\Delta_{k+1}$, we obtain a new, larger malnormal family of cyclic subgroups $$\langle h_1\alpha h_1^{-1}\rangle,\dots,\langle h_t\alpha h_t^{-1}\rangle,\langle g_{11}\gamma_1^{NN_1}g_{11}^{-1}\rangle,\dots,\langle g_{lk}\gamma_j^{NN_k}g_{lk}^{-1}\rangle$$ in $\Delta_{k+1}$. 

By the Malnormal Special Quotient Theorem/Omnipotence (see Theorem~\ref{sammy}) we can choose distinct positive integers $M_1,M_2$ which distinguish conjugates coming from $\Lambda$ from everything else. By Theorem~\ref{thm:omnirelative}, we produce a word hyperbolic and virtual special quotient
$\overline{\Delta_{k+1}}\cong \Delta_{k+1}/G$
where
\begin{align*}
    G \supset &\langle\langle h_1\alpha^{M_1} h_1^{-1},\dots, h_t\alpha^{M_1}h_t^{-1},h_{t+1}\alpha^{M_2}h_{t+1}^{-1},\dots,h_{l}\alpha^{M_2}h_{l}^{-1}, \\
    & g_{11}\gamma_1^{M_1NN_1}g_{11}^{-1},\dots,g_{t1}\gamma_1^{M_1NN_1}g_{t1}^{-1},g_{(m+1)1}\gamma_1^{M_2NN_1}g_{(m+1)1}^{-1},\\
    &\dots, g_{l1}\gamma_1^{M_1NN_1}g_{l1}^{-1}\\
    &\dots,\\
    & g_{1j}\gamma_j^{M_1NN_j}g_{1j}^{-1},\dots,g_{tj}\gamma_j^{M_1NN_j}g_{tj}^{-1},g_{(m+1)j}\gamma_j^{M_2NN_j}g_{(m+1)j}^{-1},\\
    &\dots, g_{lj}\gamma_j^{M_2NN_j}g_{lj}^{-1}
    \rangle \rangle.
\end{align*}
We then fix a finite-index $\overline{\Theta_{k+1}}<\overline{\Delta_{k+1}}$ which is torsion-free and $\overline{\Lambda}$-invariant. Its preimage in $\Delta_{k+1}$ denoted by $\Theta_{k+1}$ will be $\Lambda$-invariant and for any $g\in\Gamma_{k+1}\smallsetminus \Lambda$, $g(\Theta_{k+1})g^{-1}\ne\Theta_{k+1}$ because otherwise we get a contradiction as conjugation by $g$ induces an isomorphism $\Delta_{k+1}/\Theta_{k+1}\to\Delta_{k+1}/\Theta_{k+1}$ which sends elements of a certain order to elements of a different order (consider e.g. the image of $h_1\alpha h_1^{-1}$). 

For all the other elements $g_{lj}\in \Gamma_j\smallsetminus \Lambda$ with $j< k+1$, we can use $\Theta_j>\Theta_{k+1}$, the $\Gamma_j$-invariant subgroup 
to rule them out as follows: if $g_{lj}\Theta_{k+1}g_{lj}^{-1}=\Theta_{k+1}$, then $g_{lj}$ will switch elements of distinct $\Theta_{k+1}$-degrees where the $\Theta_{k+1}$-degree of an element is the minimal power of an element contained in $\Theta_{k+1}$. In particular, conjugating by $g_{lj}$ will take $g_{1j}\gamma_j^{NN_j}g_{1j}^{-1}$, an element with $\Theta_{k+1}$-degree $M_1$, to an element with $\Theta_{k+1}$-degree $M_2\neq M_1$.

This shows that $\Theta_{k+1}$ has $\Aut(\Theta_{k+1})\cap\Gamma_i = \Lambda$ for $i=1,\dots,k+1$ and finishes the induction step.
\end{proof}

We finish this section with the two corollaries stated in the introduction. The proof of Corollary~\ref{cor1} is immediate: By Theorem~\ref{mainthm1} and \cite[Theorems~4.1 and 4.2]{ACM} every lattice in $\PSL(2,\bbC)$ is the normalizer of infinitely many of its sublattices and it remains to apply the Bridson--Reid normalizer inheritance theorem \cite[Theorem~4.4]{BRPrasad}. The proof of Corollary~\ref{cor2} is entirely similar to the proof of Corollary~4.7 in \cite{ACM}. It is based on the virtual fibering property proved by Agol \cite{AgolHaken} and the results of Bridson--Reid \cite{BRPrasad}. We emphasize that combining the results of \cite{ACM} and this paper allows us to remove any arithmeticity related assumptions in these corollaries.

\section{Isometry groups of arithmetic hyperbolic \texorpdfstring{$n$}{n}-mani\-folds}
\label{sec:isometries}

The purpose of this section is to show how to extend the result of  \cite{BelLubotzkyIsom} to arithmetic hyperbolic $n$-manifolds. We prove that for every $n\geq 2$, every finite group is realized as the full isometry group of infinitely many arithmetic hyperbolic $n$-manifolds, either compact or non-compact. 

Similarly to Section~\ref{sec:normalizers}, the proof is based on the fact that there are only finitely many maximal lattices that have to be considered in the argument. 
A close look at the proof of the main results in \cite{ACM} and \cite{BelLubotzkyIsom} shows certain similarities of the two arguments. In both cases one needs to kill the extra symmetries by carefully choosing the set elements or subgroups that can not be preserved by them. In \cite{ACM} and Section~\ref{sec:normalizers} above this is done by using the geometric Malnormal Special Quotient Theorem, while in \cite{BelLubotzkyIsom} the corresponding tool comes from the algebraic \mbox{pro-$p$} properties of free groups \cite{Lu1}. In order to extend the proof of \cite{BelLubotzkyIsom} to the arithmetic case we need to modify this choice of the subgroups to take into account the finite collection of the maximal lattices that are now involved in the construction. 

We briefly recall the main steps of the proof and elaborate on the required extra steps. Although the main idea is simple, some of these extra steps are not obvious and require good care. The goal of this section is to extend the counting method of \cite{BelLubotzkyIsom} for proving the following theorem.

\begin{theorem}\label{mainthm2}
For every $n\ge 2$ and every finite group $G$ there exist infinitely many arithmetic hyperbolic $n$-manifolds $\calM$ of the simplest type with $\Isom(\calM) \cong G$.
\end{theorem}

As in \cite[Section~2]{BelLubotzkyIsom}, the proof of the theorem splits into an algebraic result and a geometric realization. 

\medskip\noindent\textbf{The algebraic result.}
Let $\Gamma_1$,\ldots, $\Gamma_k$, $k \ge 1$,  be finitely generated groups and $\Delta_i \trianglelefteq \Gamma_i$ finite index normal subgroups such that
$$\Delta_1 \ge \Delta_2 \ge \ldots \ge \Delta_k > M,$$
where $M$ is an infinite index normal subgroup of $\Delta_1$ with $\Delta_1/M \cong F_r$, a free group on $r\ge 2$ generators. 
We denote by $N_i = N_{\Gamma_i}(M)$ the normalizers of $M$ in $\Gamma_i$. We have 
$\Delta_i\le N_i \le \Gamma_i$ for $i = 1\ldots,k$. 

The group $N_i$ acts by conjugation on a non-abelian free group $F = \Delta_i/M$ and let $C_i = C_{N_i}(\Delta_i/M)$  denote the kernel of this action.
Then $D_i = \Delta_i C_i$ is the subgroup of all elements of $N_i$ which induce inner automorphisms of $F$. Both $C_i$ and $D_i$ are clearly normal in $N_i$.
Moreover, $M$ is normal in $D_i$, $\Delta_i$ and in $C_i$. As $F = \Delta_i/M$ has a trivial center and hence intersects $C_i/M$ trivially, $\Delta_i\cap C_i = M$. So, taking mod $M$ we get
$$D_i/M = \Delta_i/M \times C_i/M,$$
where $C_i/M$ is a finite group because $\Delta_i/M$ is of finite index in $D_i/M$:
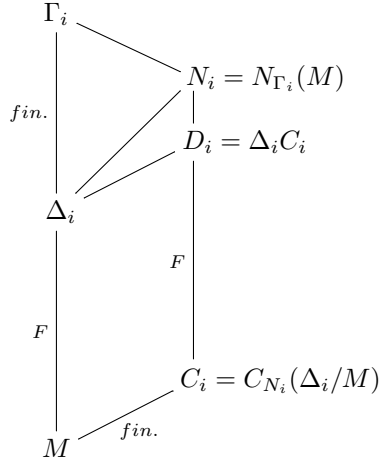
$$
\xymatrix@C=2pt@R=2.8ex{
*+[c]{\Gamma_i} \ar@{-}[ddd]_{fin.} \ar@{-}[dr] & \\
                                &*[r]{\:N_i = N_{\Gamma_i}(M)} \ar@{-}[d] \ar@{-}[ddl]  \\
                                &*+[r]{D_i = \Delta_i C_i} \ar@{-}[dl] \ar@{-}[dddd]_{F}\\
*+[r]{\Delta_i} \ar@{-}[dddd]_{F}    & \\
                                & \\
                                & \\
                                &*+[r]{C_i = C_{N_i}(\Delta_i/M)} \ar@{-}[ld]^{fin.} \\
{M} \\ }
$$
\begin{proposition}\label{prop-aut1}
Let $\Gamma_i \trianglerighteq \Delta_i$, $D_i$, for $i=1, \ldots, k$, and $M$ be as above. For every finite group $G$ there exist infinitely many finite index subgroups $A$ of $\Delta_k$ which contain $M$ and have  $N_{\Delta_i}(A)/A \cong G$  and $N_{\Gamma_i}(A) = N_{D_i}(A)$ for all $i=1, \ldots, k$.
\end{proposition}

\begin{proof}
The argument is very similar to the main algebraic result in \cite[Section~4]{BelLubotzkyIsom}. The main difference comes in the definition of a certain subgroup $K$ of $\Delta$ containing $M$. To this end we first define $K_j \trianglelefteq \Delta_j$ using \cite{Lu1} as it was done for $K \trianglelefteq \Delta$ in~\cite{BelLubotzkyIsom}. Then  we let $K$ be a finite index normal subgroup of $\Delta_1$ contained in $\bigcap_{j = 1}^{k}K_j$. Notice that since all $K_j$ contain $M$, the group $K \supset M$, and also that $K$ is a finite index normal subgroup of each $\Delta_j$.

From here the subgroup growth techniques (see \cite{LubotzkySegal}) allow us to prove the proposition the same way as in \cite{BelLubotzkyIsom}. First we show that there infinitely many subgroups $A < K$ that contain $M$ and have 
$N_{\Delta_k}(A)/A \cong G$  and $N_{\Gamma_k}(A) = N_{D_k}(A)$. Then applying the same counting argument to this collection of subgroups considered as subgroups of $K_{k-1}$ we obtain that majority of them have $N_{\Delta_{k-1}}(A)/A \cong G$  and $N_{\Gamma_{k-1}}(A) = N_{D_{k-1}}(A)$. Repeating this for the remaining $i = k-2, \ldots, 1$ finishes the proof.  This argument actually shows that most of the subgroups $A < K$ of a sufficiently large index that contain $M$ satisfy the conditions of the proposition. 
\end{proof}

\medskip\noindent\textbf{The geometric realization.} Here we need to show that in each dimension $n \ge 2$ there exist arithmetic lattices in $H = \Isom(\calH^n)$ satisfying the assumptions of Proposition~\ref{prop-aut1} and, moreover, that $\Gamma_1, \ldots, \Gamma_k$ is a complete list of maximal arithmetic lattices that contain $M$. 

\begin{proposition}\label{prop-aut2}
For every $n\ge 2$ there exist maximal cocompact
arithmetic lattices $\Gamma_1$, \ldots, $\Gamma_k$ in $H$ with subgroups $M$, $\Delta_i$ and $D_i$  satisfying the following properties:
\begin{itemize}
\item[(i)] $\Delta_i\trianglelefteq\Gamma_i$ and $[\Gamma_i : \Delta_i] < \infty$ for all $i = 1, \ldots, k$.
\item[(ii)] $\Delta_1 \ge \ldots \ge \Delta_k > M$, $\Delta_1 \triangleright M$, and $\Delta_1/M$ is a non-abelian free group.
\item[(iii)] For all $i$, $[\Gamma_i : D_i] < \infty$, $D_i\le N_{\Gamma_i}(M)$ and
\newline $D_i = \{\delta\in N_{\Gamma_i}(M) \mid \delta {\text\ induces\ an\
inner\ automorphism\ on\ } \Delta_i/M \}$.
\item[(iv)] The groups $D_i$ are torsion-free.
\item[(v)] If $\Gamma$ is a maximal lattice in $H$ containing $M$, then $\Gamma = \Gamma_i$ for some $i$. 
\end{itemize}    
\end{proposition}
\begin{proof}
The proof is based on \cite[Section~5]{BelLubotzkyIsom}. 

Let $\Gamma_0$ be an arithmetic lattice in $H$ defined by a quadratic form (also called a type I arithmetic lattice) and let $\Gamma_1$ be a maximal arithmetic lattice containing $\Gamma_0$. Then the quotient orbifold $\Gamma_0\backslash\calH^n$ has an immersed totally geodesic hypersurface, we can assume that the hypersurface is embedded and, moreover, there is a finite index normal subgroup $\Lambda_1 < \Gamma_1$ which is contained in $\Gamma_0$ and such that $\Lambda_1\backslash\calH^n$ has an embedded totally geodesic hypersurface. 

We now come to a setting when $\Lambda_1 \triangleleft \Gamma_1$ decomposes in a direct product with amalgam or an HNN-extension. 
The argument with the strong approximation and the congruence quotients in \cite[Section~5]{BelLubotzkyIsom} applies verbatim to provide the subgroups $M$, $\Delta_1$ and $D_1$ with the desired properties. 


Now the crucial remark is that $M$ is contained in only finitely maximal lattices $\Gamma_1,\ldots, \Gamma_k$ in $H$, as it follows from the classification of the maximal arithmetic subgroups (see \cite{Bel07}). Assume that we found subgroups $\Delta_1$, \ldots, $\Delta_{i-1}$ and $D_1$, \ldots, $D_{i-1}$ with the required properties. We define $\Delta_i$ to be a finite index normal subgroup of $\Gamma_i$ contained in $\Delta_{i-1}$, i.e. we have $\Delta_{i-1}\cap\Gamma_i < \Gamma_i$, a finite index subgroup, so we take the intersection of its $\Gamma_i$-conjugates. By the construction, $\Delta_i \triangleright M$. The subgroup $D_i \leq \Gamma_i$ defined as in (iii) contains $\Delta_i$ hence it has a finite index in $\Gamma_i$. It remains to show that $D_i$ is torsion-free. 

Using some notations of \cite{BelLubotzkyIsom}, we have:
$$ 
\xymatrix@R=2.8ex{
& & R \ar[dr]_{\bar\pi}\\
\Delta_1 \ar[r] & \Lambda_1 \ar[ur]_{\widetilde\pi} \ar[rr]_{\pi} & & Q \\
\Delta_i \ar@{-}[u] \\
M \ar@{-}[u] \ar[uur] \\ }
$$
where $R = Q*_TQ\ (\text{or}\ =Q*_T)$, $Q = \POm_{n+1}(\F_l) \leq \PSO_{n+1}(\F_l)$ and $T \leq \PSO_n(\F_l)$, $l$ --- a sufficiently large prime. 

The classification of the maximal arithmetic lattices implies that $\Gamma_1$, \ldots, $\Gamma_k \subset \mathrm{G}(\calO)$, where $\mathrm{G}$ is an associated algebraic group over a totally real field $\mathfrak{k}$ and $\calO$ is the ring of $S$-integers of $\mathfrak{k}$. Moreover, the finite set $S$ of the inverted primes depends only on $\mathrm{G}$ and $\mathfrak{k}$ (cf. \cite{Bel07}). 

Thus we have $\Gamma_i \subset \mathrm{G}(\calO)$ and for every $\gamma \in \Gamma_i$ we have a well defined image $\bar{\gamma} = \pi(\gamma) \in \PSO_{n+1}(\F_l)$. We are interested in the group $D_i = \Delta_iC_i$ with $C_i = \{ \delta\in N_{\Gamma_i}(M) \mid \delta|_{\Delta_i/M} = \mathrm{id} \}$. We already know that $\Delta_i \leq \Gamma_i(l)$, the congruence subgroup of $\Gamma_i$. We will show that $C_i \leq \Gamma_i(l)$ which will prove that $D_i \leq \Gamma_i(l)$ and hence $D_i$ is torsion-free by Selberg's lemma. 

Let $c \in C_i$. Then $\bar{c}$ acts on $Q$, which induces its action $\alpha$ on $R$ compatible with the trivial action of $c$ on $\Delta_i/M$. Now $R = F \rtimes Q$ and the free group $F_i = \Delta_i/M$ is a finite index subgroup of $F$. The action $\alpha|_{F_i}$ is trivial, hence $\alpha|_F = \mathrm{id}$ and by \cite[Proposition~3.4(c)]{BelLubotzkyIsom} the automorphism $\alpha$ of $R$ is trivial. It follows that $\bar{c}$ centralizes $Q$, hence $\bar{c}$ is trivial by \cite[Section~3.3]{BelLubotzkyIsom} and hence $c \in \Gamma_i(l)$.

This finishes the proof of the theorem by induction on $i$.
\end{proof}

\begin{remark}    
We note that arithmetic manifolds of type II and type III do not have codimension one totally geodesic subspaces and thus for the lattices of these types the method of this section will not apply (see \cite{BBKS} for details on geodesic subspaces of arithmetic hyperbolic orbifolds). It is an open problem if every finite group can be realized as a full group of isometries of such  manifolds which exist in every odd dimension $n \ge 3$. This problem is interesting because its solution may shed new light on the isometry groups of complex hyperbolic manifolds (cf. \cite{Lub-Stover}). 
\end{remark}

\medskip\noindent\textbf{Proof of Theorem~\ref{mainthm2}.}
Let $n\ge2$, $G$ is a finite group, $\Gamma_1, \ldots, \Gamma_k$, $\Delta_i$ and $D_i$ are lattices as in Proposition~\ref{prop-aut2}, and $M$ is a corresponding infinite index subgroup of $\Delta_1$. By Proposition~\ref{prop-aut1}, there exist infinitely many finite index subgroups $A$ of $\Delta_k$ with $N_{\Delta_i}(A)/A \cong G$  and $N_{\Gamma_i}(A) = N_{D_i}(A)$ for all $i$. 

Let $A$ be one of these groups. Notice that since $A \supset M$, by Proposition~\ref{prop-aut2}(v) the normalizer $N_H(A)$, which is known to be a lattice, is contained in $\Gamma_i$ for some $i\in I \subset \{1,\ldots, k\}$. Let $j$ be the smallest index for which it holds. We replace $A$ by $B = AC_j$. Since $A$ and $C_j$ are both in $D_j$ and $C_j\lhd D_j$, $B$ is indeed a subgroup and it is contained in $D_j$. We will show that $N_H(B)/B \cong G$ which will finish the proof of the theorem.

On one hand, if $\gamma\in N_H(A) = N_{\Gamma_j}(A)$, then it also normalizes $C_j$,
since $N_{\Gamma_j}(A)\le D_j$ and $C_j$ is normal in $D_j$. This shows that $N_H(A) \subset N_{\Gamma_j}(B)$ and hence $N_H(B)$ is contained in (some of the) maximal lattices $\Gamma_i$, $i\in I$.

Let $\gamma\in N_H(B) \subset \Gamma_i$. Then $\gamma$ also normalizes $\Delta_i$ (since $\Delta_i\lhd\Gamma_i$). Recalling that $\Delta_i \subset \Delta_j$ and $A \subset \Delta_i$, we conclude that 
$$A^\gamma = (B\cap\Delta_j)^\gamma = (B\cap\Delta_i)^\gamma = B^\gamma\cap\Delta_i^\gamma = B\cap\Delta_i = A,$$
and hence $\gamma\in N_H(A)$.

Thus we have 
$$N_H(B)/B = N_H(A)/AC \cong N_{\Delta_j}(A)/A \cong G.$$
\qed

\begin{remark}
As in \cite{BelLubotzkyIsom}, the proofs of the results in this section extend directly to non-cocompact arithmetic hyperbolic lattices. 
\end{remark}

\bibliography{main.bib}{}
\bibliographystyle{siam}

\end{document}